\documentclass{article}
\usepackage{amsmath,amsthm,amssymb}
\usepackage{cite}

\numberwithin{equation}{section}
\allowdisplaybreaks[3]

\newtheorem{theorem}{Theorem}[section]
\newtheorem{corollary}[theorem]{Corollary}
\newtheorem{proposition}[theorem]{Proposition}

\theoremstyle{remark}
\newtheorem{remark}[theorem]{Remark}
\newtheorem{definition}[theorem]{Definition}
\newtheorem{example}[theorem]{Example}

\newcommand\B{\mathcal{B}}
\newcommand\F{\mathcal{F}}
\newcommand\K{\mathcal{K}}

\newcommand\M{\mathcal{M}}
\newcommand\N{\mathbb{N}}
\renewcommand\O{\mathcal{O}}
\newcommand\R{\mathbb{R}}

\newcommand\Ga{\Gamma}
\newcommand\ga{\gamma}
\newcommand\La{\Lambda}
\newcommand\la{\lambda}
\newcommand\1{1\!\!1}

\newcommand\Bbs{B_{\mathrm{bs}}(\Ga_0)}
\newcommand\BBs{B_{\mathrm{bs}}(\Ga_0^2)}
\newcommand\ls{\mathrm{ls}}
\newcommand\Bb{\B_\mathrm{c}({X})}
\newcommand\Ob{\O_\mathrm{c}({X})}
\newcommand\Bbg{\B_\mathrm{b}(\Ga_0)}
\newcommand\BBg{\B_\mathrm{b}(\Ga_0^2)}
\newcommand\KK{\mathrm{K}}
\newcommand\Mlf{\M_\mathrm{lf}(\Ga_0)}
\newcommand\MLf{\M_\mathrm{lf}(\Ga_0^2)}
\newcommand\Mfm{\M_\mathrm{fm}^1(\Ga)}
\newcommand\MFm{\M_\mathrm{fm}^1(\Ga^2)}
\newcommand\Fcyl{\F_\mathrm{cyl}(\Ga)}
\newcommand\FCyl{\F_\mathrm{cyl}(\Ga^2)}
\newcommand\Mpi{\M_{\mathrm{fm},\pi}^1(\Ga)}
\newcommand\MPi{\M_{\mathrm{fm},\pi}^1(\Ga^2)}
\newcommand\Mp{\M^1(\Ga)}

\newcommand\cnt[2]{\text{\setbox2=\hbox{#1}\rlap{\hbox to \wd2{\hfil#2\hfil}}\box2}}
\newcommand\Star{\mathbin{\cnt{$\bigcirc$}{$\star$}}}

\newcommand\eps{\varepsilon}

\DeclareMathOperator*{\esssup}{ess\,sup}

\author{Dmitri Finkelshtein\thanks{Institute of Mathematics,
         National Academy of Sciences of Ukraine,
         01601 Kiev-4, Ukraine, e-mail:fdl@imath.kiev.ua}}
\title{Towards on convolutions on configuration spaces. II.~Spaces of locally finite configurations}

\begin{document}

\maketitle
\begin{abstract}
In the second part of the paper we consider a convolution of probability measures on spaces of locally finite configurations (subsets of a phase space) as well as their connection with the convolution of the corresponding correlation measures and functionals. In particular, the convolution of Gibbs measures is studied. We describe also a connection between invariant measures with respect to some operator and properties of the corresponding image of this operator on correlation functions.
\end{abstract}



\section{Introduction}

The present paper is the second part of the publication devoted to convolutions on spaces of configuration in continuum. The first part \cite{Fin2012a} concerned convolutions over spaces of finite configurations. More precisely, let $X$ be a connected oriented non-compact Riemannian $C^\infty$-manifold, $\mathcal{O}({X})$ be the class of all open subsets from ${X}$, $\B({X})$ be the corresponding Borel $\sigma$-algebra. We denote the classes of all open and Borel subsets from ${X}$ which have compact supports by $\Ob$ and~$\Bb$, correspondingly. Let $m$ be a non-atomic Radon measure on ${X}$, i.e., $m(\La)<\infty$, $\La\in\Bb$ and $m(\{x\})=0$, $x\in{X}$. Suppose also that there exists a sequence $\{\La_n\}_{n\in\N}\subset\Bb$ such that $\La_n\subset\La_{n+1}$, $n\in\N$ and $\bigcup_{n\in\N}\La_n=X$. The space of finite configurations over $X$ is the set
\begin{equation}\label{Ga0}
    \Ga_{0}:=\bigsqcup_{n\in \N_0}\Ga^{(n)},
\end{equation}
where $\Ga^{(n)}\simeq \widetilde{X^n}/S_n$,  $\widetilde{X^n} = \bigl\{ (x_1,\ldots ,x_n)\in X^n\bigm| x_k\neq x_l, \text{ if } k\neq l\bigr\}$, $S_n$ is a group of permutations of the set $\{1,\ldots,n\}$, and symbol ``$\sqcup$'' means a disjoint union. The more detailed description of these and sequel notations see in \cite{Fin2012a} as well as a short review of investigations in configuration spaces theory. In particular, on the space $\Ga_0$ there are naturally defined topological and measurable structures which are generated by the corresponding structures of the space $X$, in particular, the Borel $\sigma$-algebra $\B(\Ga_0)$ might be considered. The basic measure on $\bigl(\Ga_0, \B(\Ga_0)\bigr)$ is the so-called Lebesgue--Poison measure
\begin{equation}\label{laz}
    \la_z=\sum_{n=0}^\infty \frac{z^n}{n!}m^{(n)}.
\end{equation}
Here $z>0$ and the measure $m^{(n)}$ on $\Ga^{(n)}$ is generated by the measure $m^{\otimes n}$ on $X^n$. The measure $\la_z$ belongs to the space $\Mlf$ of all locally finite measures on $\Ga_0$, that means that $\la_z(B)<\infty$, for any measurable bounded domain $B$ (in symbol, $B\in\Bbg$), i.e., $B$ is such that there exist $\La\in\Bb$ and~$N\in\N$ with $B\subset \bigsqcup_{n=0}^N\Ga_\La^{(n)}$. Set $\la:=\la_1$.

For measurable functions $G_1,G_2$ on $\Ga_0$ (in symbol, $G_1,G_2\in L^0(\Ga_0)$), properties of the two following convolutions were considered in \cite{Fin2012a}:
 \begin{align}\label{ast-q}
 (G_1 * G_2)(\eta)&:=\sum_{\xi_1\sqcup\xi_2=\eta}
 G_1(\xi_1)\,G_2(\xi_2)
\\ \label{star-q}
 (G_1\star G_2)(\eta)&:=
 \sum_{\xi_1\cup\xi_2=\eta}G_1(\xi_1)\,G_2(\xi_2).
\end{align}
There was considered also a convolution of measures on $\bigl(\Ga_0, \B(\Ga_0)\bigr)$. Namely, let $\rho_1, \rho_2$ be measures on
$\bigl(\Ga_0,\B(\Ga_0)\bigr)$, then the measure $\rho:=\rho_1\ast\rho_2$ on
$\bigl(\Ga_0,\B(\Ga_0)\bigr)$ is said to be the convolution of $\rho_1$ and $\rho_2$ if, for any measurable $G:\Ga_0\to \R$,
\begin{equation}
\int_{\Ga_0} G(\eta)d\rho(\eta)
=\int_{\Ga_0} \int_{\Ga_0}G(\eta_1\cup\eta_2)
\,d\rho_1(\eta_1)\,d\rho_2(\eta_2),\label{convmeasfin}
\end{equation}
if only the right hand side is finite. There was shown in \cite{Fin2012a} also that the existence of the Radon--Nikodym derivatives $k_i=\dfrac{d\rho_i}{d\la}$, $i=1,2$ implies the existence of the such derivative $k$ for $\rho=\rho_1\ast\rho_2$ and, moreover, $k=\dfrac{d\rho}{d\la}=k_1\ast k_2$ in the sense of \eqref{ast-q}. It should be fixed also the following important property: for any $H,G_1,G_2\in L^0(\Ga_0)$ one has
\begin{equation}\label{minlosid-ast}
\int_{\Ga_0}H(\eta)(G_1\ast G_2)(\eta)d\la(\eta) =
\int_{\Ga_0}\int_{\Ga_{0}}H(\eta\cup\xi)G_1(\eta)G_2(\xi)d\la(\xi)d\la(\eta),
\end{equation}
if only at least one of integrals is well-defined (see. e.g. \cite{Kun1999}).

The space $\Ga$ of locally finite configurations and the corresponding main structures are considered in Section~2 of this paper. In~Section~3, we consider elements of the so-called harmonic analysis on configuration spaces which we will need in the sequel. In particular, this analysis is connected with properties of the convolution \eqref{star-q}. In~Section~4, we study spaces of locally finite configurations with two different types of points. This allows us to consider in~Section~5 convolutions of probability measures on spaces of locally finite configurations and their relations to the convolutions on spaces of finite configurations mentioned above. We discover also there the question about the convolution of Gibbs measures. In~Section~6, we construct a number of examples of operators which are connected with the notion of the derivation operator with respect to (w.r.t. in the sequel) the convolution \eqref{ast-q}; the latter operators were considered in \cite{Fin2012a}.

Author would like to thank Prof. Dr. Yuri Kondratiev for useful discussions. The paper was partially supported by The Ukraine President Scholarship and Grant for young scientists.

\section{Spaces of locally finite configurations}
\begin{definition}\label{def:Ga}
We consider the space of configurations $\Ga$ over ${X}$ as the set of all locally finite subsets from~${X}$, i.e.,
\begin{equation}
\Ga:=\bigl\{ \ga\subset{X}\bigm| |\ga\cap\La|<\infty \text{ for all
} \La\in\Bb\bigr\}.
\end{equation}
\end{definition}
Definition \ref{def:Ga} implies that a locally finite subset is at most countable subset of~${X}$ without finite occupation points. It is clear that $\Ga_0$ is a subset of $\Ga$, however, the space $\Ga_0$ has a special meaning to the sequel and it is considered independently. Let $\Ga_\La$ be the set of all configurations $\ga\in\Ga$ which are subsets of~$\La\in\Bb$. By Definition~\ref{def:Ga}, all the such configurations are finite. Hence, as a set, $\Ga_\La$ coincides with~$\Ga_{0,\La}$ (see a definition in~\cite{Fin2012a}).

Let $C_0({X})$ denote the class of all real-valued continuous functions on~${X}$ with compact supports. For any $f\in C_0({X})$, we define a linear function on~$\Ga$ by  $\langle f,\ga \rangle := \sum_{x\in\ga}f(x)$. It is worth noting that the sum is taken over finite set of points from~$\ga$ only whose are inside in the bounded in ${X}$ support of the function $f$. The weakest topology such that all linear functions
$\Ga\ni\ga\mapsto\langle f,\ga\rangle\in\R$, $f\in C_0({X})$ are continuous is said to be the vague topology $\O(\Ga)$ on the space $\Ga$. Note also that any configuration $\ga\in\Ga$ may be identified with the measure $\ga(\cdot):\B({X})\to \R_+$
 on~${X}$ which are a linear combination of Dirac measures, namely, $\ga\leftrightarrow\sum_{x\in\ga}\eps_x$.
 By Definition~\ref{def:Ga}, the measure $\ga$ is a Radon measure on~$\B({X})$, namely,
 \[
 \ga(A)=\int_A\,d\ga(x)=\sum_{x\in\ga}\int_A\,d\eps_x(y)=\lvert\ga\cap
 A\rvert<\infty, \qquad A\in\Bb.
 \]
Therefore, the space of configurations $\Ga$ might be isomorphically embedded into the space $\M({X})$ of Radon measures on~${X}$. Then, the vague topology $\O(\Ga)$ might be induced by the vague topology on $\M({X})$, the latter is defined in e.g.~\cite[Section 7.3]{Fol1999}.

Set $\ga_\La:=\ga\cap\La$, for any $\La\in\Bb$, $\ga\in\Ga$.
The base of topology $\O(\Ga)$ is given by
$\bigl\{\ga\in\Ga\bigm||\ga_\La|=n, \ga_{\partial\La}=\emptyset
 \bigr\}$,
where $\La\in\Bb$, $n\in\N_0$, and~$\partial\La$ is the boundary to $\La$, see e.g. \cite{Len1975}. This topology is separable and metrizable, see e.g. \cite{MKM1978}, moreover, the corresponding metric space will be complete. Note that the topology $\O(\Ga_\La)$ induced by the topology $\O(\Ga)$ does not coincide with the topology $\O(\Ga_{0,\La})$. However, the corresponding Borel $\sigma$-algebras $\B(\Ga_\La)$ and
$\B(\Ga_{0,\La})$ are coincided (see \cite{KK2002} for more details).

The Borel $\sigma $-algebra which corresponds to~$\O(\Ga)$ we denote by $\B(\Ga)$. For any $\La\in\Bb$, we consider a mapping $N_\La:\Ga\to \N_0$ in the following way: $N_\La(\ga)=|\ga_\La|$. Next, for any $\La\in\Bb$, we define a mapping $p_\La:\Ga\to \Ga_\La$, given by $p_\La(\ga):=\ga\cap\La$.
Then, $\B(\Ga)$ is the smallest $\sigma $-algebra such that all mappings $N_\La$, $\La\in\Bb$ are measurable, see e.g. \cite{AKR1998a}, i.e.,
$\B(\Ga)=\sigma\bigl( N_\La \bigm| \La\in\Bb\bigr)$.
We consider also the family of $\sigma$-algebras
$\B_\La(\Ga):=\sigma(N_{\La'}\mid\La'\in \Bb, \La'\subset\La)$, for $\La\in\Bb$.
It is worth noting that the $\sigma $-algebras $\B_\La(\Ga)$ and~$\B(\Ga_\La)$
 are~$\sigma$-isomorphic, see e.g. \cite{KK2002}, that means that there exists a bijection between them which preserves set operations, including countable unions.

Let $\mu$ be a probability measure on~$\bigl(\Ga,\B(\Ga)\bigr)$, $\Mp$ denote the class of all such measures. Let $\La\in\Bb$. The projection of $\mu^\La$ of the measure $\mu$ onto the measurable space $\bigl(\Ga_\La,\B(\Ga_\La)\bigr)$ is an image of the measure $\mu$ under the mapping $p_\La$, i.e., $\mu^\La(A):=\mu\bigl(p_\La^{-1}(A)\bigr)$, $A\in\B(\Ga_\La)$.
A probability measure $\mu$ on~$\bigl( \Ga, \B(\Ga) \bigr)$ is said to have finite local
moments of all order if $ \int_\Ga \lvert\ga_\La\rvert^n\,d\mu(\ga)<+\infty$ for any
$\La\in\Bb$ and~$n\in\N_0$. Let $\Mfm$ denote the class of all such measures.

An example of a measure with local finite moments is the Poisson measure. To define it let us fix $z>0$ and, for any $\La\in\Bb$, we consider the Lebesgue--Poisson measure $\la_z$ on $\Ga_\La=\Ga_{0,\La}$. By \eqref{laz}, $\la_z(\Ga_\La)=e^{z m(\La)}$.
We define a probability measure on~$\bigl( \Ga_\La,
\B(\Ga_\La)\bigr)$, given by $\pi_z^\La:=e^{-z m(\La)}\la_z$. For any $\La_1,\La_2\in\Bb$, $\La_1\subset\La_2$, let $p_{\La_2,\La_1}:\Ga_{\La_2}\to \Ga_{\La_1}$ be defined by
$p_{\La_2,\La_1}(\eta)=\eta_{\La_1}$, $\eta\in\La_2$.
It is easily seen that
$ \pi^{\La_2}_z\bigl(p_{\La_2,\La_1}^{-1}(A)\bigr)=\pi_z^{\La_1}(A)$, $A\in\B(\Ga_{\La_1})$,
i.e., the family of measures $\bigl\{\pi_z^\La \bigm| \La\in\Bb\bigr\}$ is consistent.
Then, by the Kolmogorov theorem version, see e.g. \cite[Theorem V.3.2]{Par1967} or \cite[Theorem 5.12]{KKdS1998}, one get that there exists a unique measure $\pi_z$ on~$\bigl(\Ga ,\B(\Ga )\bigr)$ such that, for any $\La\in\Bb$, the measure $\pi^\La_z$ is the projection of $\pi_z$ on~$\bigl(\Ga_\La,\B(\Ga_\La)\bigr)$. The measure $\pi_z$ is said to be the Poisson measure with intensity (or, parameter) $z>0$ on the space of configurations $\Ga$.

An important property of the Poisson measure is the so-called Mecke identity, which was actually proved for the Poisson processes by N.~Campbell, see~\cite{Cam1909,Cam1910}. This identity states that for any $\B(\Ga)\times\B({X})$-measurable function $h:\Ga\times{X}\to \R$
\begin{equation}\label{MeckeID}
 \int_\Ga\sum_{x\in\ga}h(\ga,x)\,dm(x)\,d\pi_z(\ga)
 =z\int_\Ga\int_{X} h(\ga\cup x,x) \, dm(x)\, d\pi_z(\ga),
\end{equation}
if only at least one of integrals is well-defined. In~\cite{Mec1968}, J.~Mecke shown that the identity \eqref{MeckeID} is a characterization to that $\pi_z$ is the Poisson emasure with the intensity $z>0$.

A measure $\mu\in\Mfm$ is said to be locally absolutely continuous w.r.t. the Poisson measure $\pi_z$, $z>0$, if, for any $\La\in\Bb$, the projection $\mu^\La$
of the measure $\mu$ on~$\bigl(\Ga_\La,\B(\Ga_\La)\bigr)$ is absolutely continuous w.r.t. the measure $\pi_z^\La$. Clearly, if $\mu\in\Mfm$ is locally absolutely continuous w.r.t. the Poisson measure $\pi_{z_0}$, for some $z_0>0$, then it is locally absolutely continuous w.r.t. the Poisson measure $\pi_z$, for any $z>0$. Let $\Mpi$ denote the class of all such measures.

Measures which are locally absolutely continuous w.r.t. the Poisson measure have properties which are similar to the properties of the Lebesgue--Poisson measure, considered in \cite{Fin2012a}. Namely, let $\mu\in\Mpi$, then for any $A\in\B({X})$ with $m(A)=0$ the following equality holds true (see e.g. \cite{Kun1999}):
 \[
 \mu\bigl(\{\ga\in\Ga \mid \ga\cap
 A\neq\emptyset\}\bigr)=0.
 \]
As a corollary, for any $\ga'\in\Ga$, $x\in{X}$,
 \[
 \mu\bigl(\{\ga\in\Ga \mid x\in\ga\}\bigr)
 =\mu\bigl(\{\ga\in\Ga \mid \ga'\cap\ga\neq\emptyset\}\bigr)=0.
 \]
Next, the set of pairs
 $\bigl\{(\ga,\ga')\in\Ga\times\Ga\mid \ga\cap\ga'\neq\emptyset\bigr\}$
is a zero set for the measure $\mu\otimes\mu$, as well as the set of pairs
 $\bigl\{(\ga,x)\in\Ga\times{X}\mid x\notin\ga\bigr\}$
is a zero set for the measure $\mu\otimes m$. Finally, it can be easily shown that that~$\Ga_0\in\B(\Ga)$ and $\mu(\Ga_0)=0$.

\begin{remark}\label{remPi}
By the definition of the Lebesgue--Poisson measure, one has that, for $z_1\neq z_2$, the measures $\la_{z_1}$ and~$\la_{z_2}$ on~$\bigl(\Ga_0,\B(\Ga_0)\bigr)$ are equivalent (i.e., they are mutually absolutely continuous). However, two measures on $\bigl(\Ga,\B(\Ga)\bigr)$ which are locally absolutely continuous w.r.t. the same Poisson measure may not be mutually absolutely continuous. Even two Poisson measures $\pi_{z_1}$
and $\pi_{z_2}$ (they are mutually locally absolutely continuous), for $z_1\neq z_2$ are orthogonal on the whole space $\Ga$ (see \cite{Sko1957} and generalization in
\cite{Tak1990}).
\end{remark}

A function $F:\Ga\to \R$ is said to be cylindric if there exists $\La\in\Bb$ such that $F$ is a $\B_\La(\Ga)$-measurable function. This property may be characterized by the equality
$F(\ga)=F\!\upharpoonright\!_{\Ga_\La}(\ga_\La)$. Let $\Fcyl$ denote the class of all cylindric functions on~$\Ga$.

\section{Elements of harmonic analysis}
Recall that (see e.g. \cite{Fin2012a} for more details) a function $G\in L^0(\Ga_0)$ has a local support if there exists $\La\in\Bb$ such that $G\upharpoonright_{\Ga_0\setminus \Ga_{0,\La}}=0$. Let $L_\ls^0(\Ga_0)$ denote the class of all measurable functions on~$\Ga_0$ with local supports. By analogy, a $G\in L^0(\Ga_0)$ has a bounded support if there exists $B\in\Bbg$ with $G\upharpoonright_{\Ga_0\setminus B}=0$. Let $\Bbs$ denote the class of all bounded measurable functions on~$\Ga_0$ with bounded supports.

Let us consider the mapping $K:L_\ls^0(\Ga_0)\to \Fcyl$ given by (see \cite{KK2002,Len1975a,Len1975})
 \begin{equation}\label{defK}
 (KG)(\ga ):=\sum_{\eta\Subset\ga}G(\eta), \qquad \ga\in\Ga,
 \end{equation}
where~$G\in L_\ls^0(\Ga_0)$. Here and subsequently, for an infinite configuration $\ga\in\Ga$, the notation $\eta\Subset\ga$ means that $\eta\subset\ga$ and $\eta\in\Ga_0$, i.e., $\eta$ is a finite subset of a set $\ga$.
It is worth noting that the summation in \eqref{defK} is taken over finite collection of all subsets from~$\ga_\La$ where $\La$ is the local support of the function $G\in L_\ls^0(\Ga_0)$.

The mapping $K:L_\ls^0(\Ga_0)\to \Fcyl$ is linear, positive preserving and has an inverse (see e.g. \cite[Proposition~3.5]{KK2002})
\begin{equation}
(K^{-1}F)(\eta):=\sum_{\xi\subset\eta}(-1)^{|\eta\setminus\xi
|}F(\xi),\qquad\eta\in\Ga_0.\label{k-1trans}
\end{equation}
In the same reference it was shown that $F\in\Fcyl$ yields $K^{-1}F\in L_\ls^0(\Ga_0)$. It should be underline, that the right hand side of \eqref{k-1trans} gives a well-defined measurable function on~$\Ga_0$ for any measurable function $F$, if only the latter is well defined on whole~$\Ga$, or even on some subspace from ~$\Ga$, which contains $\Ga_0$.
\begin{proposition}[\!\!{\cite[Proposition~3.5]{KK2002}}]\label{pb}
Let $G\in\Bbs$. Then $KG\in\Fcyl$, moreover, there exist $C>0$, $\La\in\Bb$, and~$N\in\N_0$ such that, for $F=KG$,
 \begin{equation}\label{polbdd}
 \lvert F(\ga)\rvert\leq C\bigl(1+|\ga_\La|\bigr)^N,
 \qquad \ga\in\Ga.
 \end{equation}
\end{proposition}
Let $\F_{\mathrm{pb}}(\Ga)$ be the class of all cylindric functions $F:\Ga\to\R$ which satisfy the inequality \eqref{polbdd}.

The latter functions is said to be (locally) polynomially bounded on~$\Ga$. As a result, $\F_{\mathrm{pb}}(\Ga)\subset L^1(\Ga,\mu)$, for any $\mu\in\Mfm$. This yields the correctness of the following definition (see \cite[Remark~4.7]{KK2002} for more details).
\begin{definition}
 Let $\mu\in\Mfm$. A measure $\rho_\mu\in\Mlf$ given by
 \begin{equation}\label{cormeas}
 \rho_\mu(A):=\int_\Ga(K\1_A)(\ga)\,d\mu(\ga), \qquad
 A\in\Bbg
 \end{equation}
 is said to be the correlation measure which corresponds to $\mu$. Here
 $\1_A:\Ga_0\to \R$ is the indicator function of a set $A\in\B(\Ga_0)$.
Note that $\rho_\mu(\{\emptyset\})=1$.
\end{definition}

\begin{proposition}[\!\!{\cite[Corollary 4.6]{KK2002}}]
Let $\mu\in\Mfm$. Then, for all $G\in \Bbs$, $G\in L^1(\Ga_0,\rho_\mu)$ and, moreover, \begin{equation}\label{GintEqKGint}
\int_{\Ga_0}G(\eta)\,d\rho_\mu(\eta )=\int_\Ga(KG)(\ga)\,d\mu(\ga).
\end{equation}
\end{proposition}

\begin{remark}
For any nonnegative $\B(\Ga_0)$-measurable function $G:\Ga_0\to \R_+:=[0;\,+\infty)$
the right hand side of \eqref{defK} defines a $\B(\Ga)$-function
$KG:\Ga\to [0;+\infty]$.
In this case the equality \eqref{GintEqKGint} holds true as well
(see~\cite[Corollary 4.6]{KK2002}).
\end{remark}

\begin{remark}
It was shown in \cite[Theorem 4.11]{KK2002}) that, for $\mu\in\Mfm$ and
 $G\in L^1(\Ga_0,\rho_\mu)$, the right hand side of \eqref{defK} is
$\mu$-a.\,s. absolutely convergent series, moreover, $KG\in L^1(\Ga,\mu)$ and the equality \eqref{GintEqKGint} holds true.
\end{remark}

\begin{proposition}[\!\!{\cite[Proposition~4.14]{KK2002}}]\label{RNder}
Let $\mu\in\Mpi$. Then the correlation measure $\rho_\mu$ is absolutely continuous w.r.t. the Lebesgue--Poisson measure $\la$ on $\bigl(\Ga_0,\B(\Ga_0)\bigr)$.
\end{proposition}

Let $\mu\in\Mfm$. Suppose that the correlation measure $\rho_\mu$ is absolutely continuous w.r.t. the Lebesgue--Poisson measure $\la$ on $\bigl(\Ga_0,\B(\Ga_0)\bigr)$. Then the corresponding Radon--Nikodym derivative $k_\mu(\eta):=\dfrac{d\rho_\mu}{d\la}(\eta)$, $\eta\in\Ga_0$ is said to be the correlation functional of the measure $\mu$. By~Proposition~\ref{RNder}, for any measure $\mu\in\Mpi$, the corresponding correlation functional always exists. Clearly,
$k_\mu(\emptyset)=1$. Functions $k_{\mu}^{(n)}:{X}^n\to \R_+$ given by
 \[
 k_{\mu}^{(n)}(x_1,\ldots,x_n):=
 \begin{cases}
 k_{\mu}(\{x_1,\ldots,x_n\}),
 & \text{if } (x_1,\ldots,x_n)\in\widetilde{{X}^n},\\
 0, & \text{otherwise}
 \end{cases}
 \]
are said to be the correlation functions of $\mu$.
 It is worth noting, that if $k_\mu$ is the correlation functional of a measure $\mu\in\Mfm$, then one can rewrite \eqref{GintEqKGint} in the following form
 \begin{equation}\label{GintEqKGint-cf}
 \int_{\Ga_0}G(\eta)k_\mu(\eta )\,d\la(\eta)
 =\int_\Ga(KG)(\ga)\,d\mu(\ga).
 \end{equation}
 Sometimes, the correlation functions are defined via equality
 \eqref{GintEqKGint-cf}. Namely, a sequence of measurable symmetric nonnegative functions
 $k_{\mu}^{(n)}:{X}^n\to \R_+$, $n\in\N$, $k_\mu^{(0)}:=1$
 is said to be a system of correlation functions corresponding
 to $\mu\in\Mfm$ if, for any $n\in\N_0$ and any measurable symmetric
 nonnegative $f^{(n)}:{X}^n\to \R_+$, the following equality holds
 \begin{multline*}
 \int_\Ga\sum_{\{x_1,\ldots,x_n\}\subset\ga}f^{(n)}(x_1,\ldots,x_n)\,d\mu(\ga)
 \\=\frac{1}{n!}\int_{{X}^n}f^{(n)}(x_1,\ldots,x_n)k_\mu^{(n)}(x_1,\ldots,x_n)
 \, dm(x_1)\ldots dm(x_n).
 \end{multline*}
In particular, for any $\mu\in\Mpi$, there exists a system of correlation functions.

Let $C>0$ and~$\delta\geq 0$. We consider the following Banach space
\begin{equation}\label{Kca}
\K_{C,\,\delta} = \bigl\{
k:\Ga_0\to \R \bigm| |k(\eta)|\leq \mathrm{const}\cdot C^{|\eta|}
(|\eta| !)^\delta \text{ for } \la\text{-a.a. } \eta\in\Ga_0 \bigr\}
\end{equation}
with the norm $\|k\|_{{C,\,\delta}}:=\esssup_{\eta\in\Ga_0}
\dfrac{|k(\eta)|}{C^{|\eta|} (|\eta| !)^\delta}$.
Clearly, for $C'\geq C$, $\delta'\geq\delta$, the inclusion
$\K_{C,\,\delta}\subset\K_{C',\,\delta'}$ holds. For $\delta=0$, we will omit the index, namely, $\K_C:=\K_{C,0}$.

Let, for arbitrary $C>0$, $\delta\geq0$, $k\in\K_{C,\,\delta}$.
Then $\Bbs\subset L^1(\Ga_0, |k|\,d\la)$. If, additionally, $\delta\in[0;1)$, then~$e_\la(f)\in L^1(\Ga_0, |k|\,d\la)$, for any $f\in L^1({X},dm)$.
The proofs of the both statement are straightforward consequences of definition \eqref{Kca}.

As a result, if $\mu\in\Mfm$ and $k_\mu$ is the correlation functional of $\mu$, such that $k_\mu\in\K_{C,\,\delta}$, $C>0$, $\delta\in[0;1)$, then, for any $f\in L^1({X},dm)$,
 \begin{equation}\label{Kexplp}
 \bigl(K e_\la(f)\bigr)(\ga)=\prod_{x\in\ga}\bigl(1+f(x)\bigr) \quad \text{for $\mu$-a.a. $\ga\in\Ga$.}
 \end{equation}
In particular, the infinite product in the right hand side of~\eqref{Kexplp} is absolutely convergent, for $\mu$-a.a. $\ga\in\Ga$. It should be pointed out, that, for $f\in C_0({X})\subset L^1({X},dm)$, the equality \eqref{Kexplp} holds true for all $\ga\in\Ga$.

The following statement describes relations between correlation functions of a measure and their projections.
\begin{proposition}[\!\!{\cite[Propositions 4.14, 4.16]{KK2002}}]
Let $\mu\in\Mpi$. Then, for any $\La\in\Bb$, $z>0$,
 \begin{equation}\label{cfoverproj}
 k_\mu(\eta)=\int_{\Ga_\La}\frac{d\mu^\La}{d\pi_z^\La}(\eta\cup\ga)\,
 d\pi_z^\La(\ga)\quad \text{for $\la$-a.a. $\eta\in\Ga_\La$.}
 \end{equation}
If, additionally, $\int_{\Ga_\La} 2^{|\eta|} k_\mu(\eta) \,d\la(\eta)<\infty$, for all $\La\in\Bb$, then
 \begin{equation}\label{projovercf}
 \frac{d\mu^\La}{d\pi_z^\La}(\ga)=
 \int_{\Ga_\La}(-1)^{|\eta|}k_\mu(\ga\cup\eta)\,
 d\la(\eta) \quad \text{for $\pi_z^\La$-a.a. $\ga\in\Ga_\La$.}
 \end{equation}
\end{proposition}

One may formulate an inverse problem: about a possibility to recover a measure $\mu\in\Mfm$ by the given system of symmetric measurable functions $k^{(n)}$ such that $k^{(n)}=k_\mu^{(n)}$. This problem may be solve by the following A.\,Lenard's results.
\begin{definition}
A function $k:\Ga_0\to \R$ is said to be positive definite in the sense of Lenard if for any $G\in\Bbs$ with~$(KG)(\ga)\geq0$, for all $\ga\in\Ga$, the following holds $\int_{\Ga_0} G(\eta) k(\eta)\, d\la(\eta)\geq 0$.
\end{definition}
It is worth noting that if $k_\mu$ is the correlation functional of a measure $\mu\in\Mfm$ then \eqref{GintEqKGint-cf} yields that $k_\mu$ is positive definite in the sense of Lenard.

\begin{proposition}\label{exuniqmeas}
Let $k:\Ga_0\to \R$ be measurable.
 \begin{enumerate}
 \item{\rm \cite[Theorem 4.1]{Len1975a}}
 Suppose that $k$ is positive definite in the sense of Lenard and the normalized condition $k(\emptyset)=1$ holds true.
 Then, there exists at least one measure $\mu\in\Mfm$ such that $k$ is the correlation functional of the measure $\mu$.

 \item {\rm \cite[Theorem 2]{Len1973}} For any $n\in\N$, $\La\in\Bb$, set
 \begin{equation}\label{snLa}
    s_n^\La:= \frac{1}{n!}\int_{\La^n}k^{(n)}(x_1,\ldots,x_n)
 \,dm(x_1)\ldots dm(x_n).
 \end{equation}
 If only $\sum_{n\in\N}\bigl(s_{n+m}^\La\bigr)^{-\frac{1}{n}}=\infty$, for all $m\in\N$, $\La\in\Bb$, then there exists at most one measure $\mu\in\Mfm$ such that $k$ is the correlation functional of the measure $\mu$.
 \end{enumerate}
\end{proposition}
  It should be pointed out, that, that, in \cite{Len1975a,Len1973}, the more wide space than $\Ga$ was considered (the so-called space of multiple configurations). An adaptation of \cite[Theorem 4.1]{Len1975a} to the case of $\Ga$ was realized in \cite[Theorem 4.4.1]{Kun1999}. The statement \cite[Theorem 2]{Len1973} is fulfilled for smaller space $\Ga$ obviously.

\begin{corollary}
  Let $k:\Ga_0\to \R$ be a positive definite function in the sense of Lenard, and suppose that the normalized condition $k(\emptyset)=1$ is holds true. Suppose that there exists $C>0$ such that $k\in\K_{C,2}$. Then there exists a unique measure $\mu\in\Mfm$ such that $k$ is the correlation functional of the measure $\mu$.
\end{corollary}
\begin{proof}
By~\eqref{snLa}, since $k\in\K_{C,2}$ one has, for all $n\in\N$,
 $s^\La_n\leq \mathrm{const} \cdot \bigl(Cm(\La)\bigr)^n n!.$
 Therefore, for all $l\in\N$,
 \begin{align*}
 \sum_{n\in\N}\bigl(s_{n+l}^\La\bigr)^{-\frac{1}{n}}
 \geq &\, \mathrm{const} \cdot\sum_{n\in\N}
 \bigl(Cm(\La)\bigr)^{-\frac{n+l}{n}}
 \bigl((n+l)!\bigr)^{-\frac{1}{n}}\\
 \geq& \, \mathrm{const} \cdot\bigl(Cm(\La)\bigr)^{-(1+l)}
 \sum_{n\in\N} (n!)^{-\frac{1}{n}}=\infty,
 \end{align*}
 which proves the assertion.
\end{proof}

The following statement shows that the~$K$-transform is a combinatorial Fourier transform w.r.t. the $\star$-convolution.
\begin{proposition}[\!\!{\cite[Proposition 3.11]{KK2002}}]\label{Fourier}
 Let $G_1, G_2\in L_\ls^0(\Ga_0)$. Then
 \begin{equation}\label{eq:Fourier}
 \bigl(K(G_1\star G_2)\bigr)(\ga)
 =(KG_1)(\ga)\cdot (KG_2)(\ga), \qquad\ga\in\Ga.
 \end{equation}
\end{proposition}

\begin{remark}
 Let $\mu\in\Mfm$ and $k_\mu$ be the correlation functional of the measure $\mu$.
 It was shown in~\cite[Lemma 4.12]{KK2002}, that the equality \eqref{eq:Fourier} holds for $\mu$-a.a. $\ga\in\Ga$ if only at least one of the following assumptions is fulfilled: 1)~$G_1,G_2\geq0$; 2)~$|G_1|\star|G_2|\in
 L^1(\Ga_0, k_\mu \,d\la)$; 3)~$G_1,G_2\in L^1(\Ga_0, k_\mu \,d\la)$.
\end{remark}

The next corollary is a direct consequence of Proposition~\ref{Fourier},
\cite[Proposition~3.5]{KK2002}, and the obvious observation that the class of functions
$\Fcyl$ is closed w.r.t. a product.
\begin{corollary}
 Let $F_1, F_2\in\Fcyl$. Then
 \begin{equation}\label{eq:invFourier}
 \bigl(K^{-1}(F_1\cdot F_2)\bigr)(\eta)
 =\bigl((K^{-1} F_1)\star (K^{-1} F_2)\bigr)(\eta),\qquad\eta\in\Ga_0.
 \end{equation}
\end{corollary}

\begin{remark}\label{rem:fourier}
 We stress that the equality \eqref{eq:invFourier} holds for any functions $F_1,F_2$ which are defined on some subset of the configuration space $\Ga$ that contains the set of finite configurations $\Ga_0$.
\end{remark}

\begin{remark}\label{rem:posdefstar}
 Let $G\in\Bbs$. By Proposition \ref{Fourier},
 $K(G\star G)=|KG|^2\geq0$. Therefore, if a function $k$ is positive definite in the sense of Lenard, then it is positive definite in the sense of the $\star$-convolution (for definition see \cite{Fin2012a}).
 It was shown in \cite{KK2002}, that if $k\in K_{C,\,\delta}$,
 $C>0$, $\delta\in[0;1)$, $k(\emptyset)=1$, and $k$ is positive definite in the sense of $\star$-convolution, then there exists a unique measure $\mu\in\Mfm$ such that $k$ is its  correlation functional.
\end{remark}

\section{Spaces of configurations of different types}
This Section is devoted to spaces of configurations of two different types, which we denote by ``$+$'' and ``$-$''. We will need the following notations. Let us consider two copies of $\Ga$, $\Ga^{+}:=\Ga$ and $\Ga^{-}:=\Ga$, and we set $\Ga^{2}:=\Ga^{+}\times\Ga^{-}$. Let $\O(\Ga^2)$ be the product-topology on $\Ga^2$. We denote the corresponding Borel $\sigma$-algebra by $\B(\Ga^2)$. Let $\M^1(\Ga^2)$ denote the class of all probability measures on~$\bigl(\Ga^2, \B(\Ga^2)\bigr)$.
Similarly, we consider, for $Y\in\B({X})$, $n\in\N$,
    \[
        \Ga_{0,Y}^{+,(n)}:=\Ga_{0,Y}^{-,(n)}:=\Ga_{0,Y}^{(n)},
        \quad\Ga_{0,Y}^+:=\Ga_{0,Y}^-:=\Ga_{0,Y},\quad
        \Ga_{0,Y}^2:=\Ga_{0,Y}^+\times\Ga_{0,Y}^-,
    \]
and one can define the corresponding product-topology $\O(\Ga^2_{0,Y})$ and Borel $\sigma$-algebra $\B(\Ga^2_{0,Y})$.
As above, for $Y={X}$, we will omit $Y$ in subscript, and, for $Y\in\Bb$, we will omit $0$ in subscript, namely, $\Ga_0^2:=\Ga_{0,{X}}^2$, $\Ga_\La^2:=\Ga_{0,\La}^2$.
It is obvious that
    \begin{align*}
        \B(\Ga^2)&=\sigma\bigl(B^+\times B^- \bigm| B^\pm\in\B(\Ga^\pm)\bigr);\\
        \B(\Ga_{0,Y}^2)&=\sigma\bigl(B^+\times B^- \bigm| B^\pm\in\B(\Ga_{0,Y}^\pm)\bigr), \qquad Y\in\B({X}).
    \end{align*}

We define also some notions similar to one-type configuration spaces. For any $\La^\pm\in\Bb$, we consider a mapping
    $\mathrm{p}_{\La^+,\La^-}:\Ga^2\to \Ga^+_{\La^+}\times\Ga^-_{\La^-}$
    given by
    $\mathrm{p}_{\La^+,\La^-}(\ga^+,\ga^-):=\bigl(\ga^+_{\La^+},
        \ga^-_{\La^-}\bigr)$, $\ga^\pm\in\Ga^\pm$.
    The projection of $\mu\in\M^1(\Ga^2)$ on
    $\bigl(\Ga^+_{\La^+}\times\Ga^-_{\La^-},\B(\Ga^+_{\La^+}\times\Ga^-_{\La^-})\bigr)$
    is the measure $\mu^{\La^+,\La^-}$ given by
    $\mu^{\La^+,\La^-}(A):=\mu\bigl(\mathrm{p}^{-1}_{\La^+,\La^-}(A)\bigr)$,
    $A\in\B(\Ga^+_{\La^+}\times\Ga^-_{\La^-})$.
A measure $\mu\in\M^1(\Ga^2)$ is said to be locally absolutely continuous w.r.t. the measure $\pi_z\otimes\pi_z$, $z>0$ if, for any $\La^\pm\in\Bb$, the projection $\mu^{\La^+,\La^-}$ of $\mu$ is absolutely continuous w.r.t. the measure
$\pi_z^{\La^+}\otimes\pi_z^{\La^-}$
on $\bigl(\Ga^+_{\La^+}\times\Ga^-_{\La^-}, \B(\Ga^+_{\La^+}\times\Ga^-_{\La^-})\bigr)$.
In the case $\La^+=\La^-=\La\in\Bb$, we will use the notations $\mathrm{p}_\La,\mu^\La,\Ga_\La^2$ instead of $\mathrm{p}_{\La,\La},
    \mu^{\La,\La},\Ga^+_{\La}\times\Ga^-_{\La}$, correspondingly. The following statements generalize the properties of measures from~$\Mpi$. They were proved by the author in~\cite{Fin2009}.
\begin{proposition}\label{goodset}
Let $\mu\in\M^1(\Ga^2)$ be locally absolutely continuous w.r.t. $\pi_z\otimes\pi_z$, $z>0$. Let also $A\in\B({X})$, $m(A)=0$. Then
\begin{align*}
&\mu\bigl(\bigl\{ (\ga^+,\ga^-)\in\Ga^2\bigm|
\ga^+\cap\ga^-=\emptyset\bigr\}\bigr)=1;\\
&\mu\bigl(\bigl\{ (\ga^+,\ga^-)\in \Ga^2 \bigm| \ga^-\cap A\neq\emptyset
\bigr\}\bigr)=0;\\
 &(\mu\otimes m)\bigl(\bigl\{ (\ga^+,\ga^-,x)\in\Ga^2\times{X} \bigm| x\in\ga^+ \bigr\}\bigr)=0.
\end{align*}
\end{proposition}

Let $\mu\in\M^1(\Ga^2)$. The marginal distributions of the measure $\mu$ are the probability measures $\mu^\pm$ on~$\bigl( \Ga^\pm, \B(\Ga^\pm)\bigr)$ given by
\begin{equation}
 \mu^\pm(A^\pm):=\int_{A^\pm}\int_{\Ga^\mp}\,d\mu(\ga^+,\ga^-),
 \qquad A^\pm\in\B(\Ga^\pm) .\label{defmarg}
\end{equation}
Let $\La\in\Bb$ and $\mu^\La$ be the projection of $\mu\in\M^1(\Ga^2)$ on $\Ga_\La^2$. The marginal distributions of $\mu^\La$ are the probability measures $(\mu^\La)^\pm$ on~$\bigl( \Ga_\La^\pm,
  \B(\Ga_\La^\pm)\bigr)$ which are defined analogously to~\eqref{defmarg}.
On the other hand, one can consider the projections $(\mu^\pm)^\La$ of $\mu^\pm$ on $\bigl( \Ga_\La^\pm, \B(\Ga_\La^\pm)\bigr)$ according to~\eqref{defmarg}.
In \cite{Fin2009}, the author shown that, for any $\La\in\Bb$,
\begin{equation}
(\mu^\pm)^\La=(\mu^\La)^\pm.\label{comuteprojmarg}
\end{equation}

We proceed now to definitions of other objects which are analogous to considered before in the one-type case.
  A function ${G}:\Ga_0^2\to \R$ has a local support if there exists $\La\in\Bb$
  such that ${G}\upharpoonright_{\Ga_0^2\setminus(\Ga_\La^+\times
    \Ga_\La^-)}=0$. Let $L^0_\ls(\Ga_0^2)$ denote the class of all measurable functions on $\Ga_0^2$ with local supports.
    A set ${B}\in \B(\Ga^2_0)$ is said to be bounded if there exist $\La\in\Bb$ and~$N\in\N$, such that
  ${B}\subset\bigl(\bigsqcup_{n=0}^N\Ga_\La
    ^{+,(n)}\bigr)\times\bigl(\bigsqcup_{n=0}^N\Ga_\La^{-,(n)}\bigr)$.
  Let $\BBg$ denote the class of all bounded sets in~$\B(\Ga^2_0)$.
A function ${G}:\Ga_0^2\to \R$ has a bounded support if there exists ${B}\in\BBg$ such that ${G}\upharpoonright_{\Ga_0^2\setminus\widetilde{B}}=0$.
   Let $\BBs$ denote the class of all bounded functions with bounded supports.
A measure ${\rho}$ on $\bigl(\Ga^2_0,\B(\Ga^2_0)\bigr)$
   is sid to be locally finite if ${\rho}({B})<\infty$, for all
   ${B}\in\BBg$. Let $\MLf$ denote the class of all such measures.
A measure $\mu\in\M^1(\Ga^2)$ has finite local moments of all orders if
$\int_{\Ga^2} |\ga^+_\La|^n |\ga^-_\La|^n \,d{\mu}
      (\ga^+,\ga^-)<\infty$ for all $\La\in\Bb$ and~$n\in\N_0$.
Let $\MFm$ denote the class of all probability measures on~$\Ga^2$ with finite local moments of all orders.
Let $\FCyl$ denote the class of measurable functions on~$\Ga^2$ which are cylindric by both variables.

\begin{definition}
 We consider the transformation
 $\KK:L_\ls^0(\Ga_0^2)\to \FCyl$ given by
 \begin{equation}\label{defK2}
 (\KK G)(\ga^+,\ga^-):=\sum_{\substack{\eta^+\Subset\ga^+\\
 \eta^-\Subset\ga^-}}G(\eta^+,\eta^-), \qquad (\ga^+,\ga^-)\in\Ga.
 \end{equation}
\end{definition}
    Let $I^\pm$ be unit operators (identical mappings) on functions on~$\Ga^\pm$ (and, hence, on $\Ga_0^\pm$).
    We define the following operators on functions on~$\Ga_0^2$: $K^+:=K\otimes I^-$, $K^-:=I^+\otimes K$.
    Then one can rewrite \eqref{defK2}:
    \begin{equation}
        \KK =K^+K^-=K^-K^+. \label{redefK2}
    \end{equation}
\eqref{redefK2} yields that
$\KK:L_\ls^0(\Ga_0^2)\to \FCyl$ is a linear mapping which preserves positive functions and has an inverse transform
\begin{align}
(\KK ^{-1}
F)(\eta^+,\eta^-)&=(K^+)^{-1}(K^-)^{-1}=(K^-)^{-1}(K^+)^{-1}\notag\\&=
\sum_{\substack{ \xi^+\subset\eta^+
\\ \xi^-\subset\eta^-}} (-1)^{\vert \eta^+\setminus\xi^+\vert + \vert
\eta^-\setminus\xi^-\vert}F(\xi^+,\xi^-), \quad (\eta^+,\eta^-)\in
\Ga ^2_0.\label{k2-1trans}
\end{align}

\begin{definition}
Let $\mu\in\MFm$. The correlation measure, corresponding to $\mu$, is the measure $\rho_\mu\in\MLf$ given by
 \begin{equation}\label{cormeas2}
 \rho_\mu(A):=\int_{\Ga^2}(\KK\1_A)(\ga^+,\ga^-)\,d\mu(\ga^+,\ga^-), \qquad
 A\in\BBg,
 \end{equation}
 where~$\1_A:\Ga_0^2\to \R$ is the indicator-function of a set $A\in\B(\Ga_0^2)$.
 Clearly, $\rho_\mu(\{\emptyset\}, \{\emptyset\})=1$.
 \end{definition}

The following propositions can be proved by analogy with the corresponding statements for the space $\Ga$.
\begin{proposition}\label{prop:intK2G}
Let $\mu\in\MFm$. Then $G\in \BBs$ yields $G\in
L^1(\Ga_0^2,\rho_\mu)$, moreover,
\begin{equation}\label{GintEqKGint2}
\int_{\Ga_0^2}G(\eta^+,\eta^-)\,d\rho_\mu(\eta^+,\eta^-)=\int_{\Ga^2}(\KK G)(\ga^+,\ga^-)\,d\mu(\ga^+,\ga^-).
\end{equation}
\end{proposition}

\begin{proposition}\label{RNder2}
Let a measure $\mu\in\MFm$ be locally absolutely continuous w.r.t. the measure
 $\pi_z\otimes\pi_z$, $z>0$. Let $\MPi$ denote the class of all such measures. Then the correlation measure $\rho_\mu$ is absolutely continuous w.r.t. the measure $\la^2:=\la\otimes\la$ on
 $\bigl(\Ga_0^2, \B(\Ga_0^2)\bigr)$.
The correlation functional of the measure $\mu$ is the corresponding Radon--Nikodym derivative
 \[
 k_\mu(\eta^+,\eta^-):=\frac{d\rho_\mu}{d\la^2}(\eta^+,
 \eta^-),\qquad(\eta^+,\eta^-)\in\Ga_0^2.
 \]
Then $k_\mu(\emptyset,\emptyset)=1$ and, for any $\La\in\Bb$ and $\la^2$-a.a. $(\eta^+,\eta^-)\in\Ga_\La^2$,
\begin{equation}\label{gencorfunc}
k_\mu(\eta^+,\eta^-)=\int_{\Ga^+_{\La}}\int_{\Ga^-_{\La}}
\frac{d\mu^{\La}}{d\la^2} (\eta^+\cup\xi^+,\eta^-\cup\xi^-)
d\la(\xi^+)d\la(\xi^-).
\end{equation}
\end{proposition}

\begin{proposition}
Let $\mu\in\MPi$. Then the marginal measures (distributions) $\mu^\pm$ belongs to $\Mpi$. Moreover, if $k_\mu$, $k_\mu^\pm$ are the correlation functionals of the measures $\mu$, $\mu^\pm$, correspondingly, the, for $\la$-a.a. $\eta^\pm\in\Ga^\pm$,
\begin{equation}\label{cf2onemptyset}
    k_\mu(\eta^+,\emptyset)=k^+_\mu(\eta^+), \qquad
    k_\mu(\emptyset,\eta^-)=k^-_\mu(\eta^-).
\end{equation}
\end{proposition}

\begin{proof}
Let a set $\La\in\Bb$ and a measurable function
$F:\Ga_\La^+\to \R_+$ be arbitrary. By~\eqref{comuteprojmarg}, for all $z>0$, one has
\begin{align*}
  \int_{\Ga_\La^+}F(\ga^+)\, d(\mu^+)^\La (\ga^+) &=
  \int_{\Ga_\La^+}F(\ga^+)\, d(\mu^\La)^+ (\ga^+) =
  \int_{\Ga_\La^+}F(\ga^+) \int_{\Ga_\La^-}\,d\mu^\La (\ga^+,\ga^-)\\&=\int_{\Ga_\La^+}F(\ga^+) \int_{\Ga_\La^-}\frac{d\mu^\La}{d(\pi_z^\La\otimes\pi_z^\La)} (\ga^+,\ga^-)\,d(\pi_z^\La\otimes\pi_z^\La) (\ga^+,\ga^-)\\&=\int_{\Ga_\La^+}F(\ga^+) \biggl(\,\int_{\Ga_\La^-}\frac{d\mu^\La}{d(\pi_z^\La\otimes\pi_z^\La)} (\ga^+,\ga^-)\,d\pi_z^\La(\ga^-)\biggr)d\pi_z^\La(\ga^+).
\end{align*}
Therefore, the measure $\mu^+$ is locally absolutely continuous w.r.t. the Poisson measure $\pi_z$, $z>0$ and, moreover, for $\pi_z^\La$-a.a. $\ga^+\in\Ga^+_\La$,
\begin{equation}\label{RNprj}
  \frac{d(\mu^+)^\La}{d\pi_z^\La}(\ga^+)=
  \int_{\Ga_\La^-}\frac{d\mu^\La}{d(\pi_z^\La\otimes\pi_z^\La)} (\ga^+,\ga^-)\,d\pi_z^\La(\ga^-).
\end{equation}
The obvious fact that, for all $\La\in\Bb$, $n\in\N$,
\[
\int_{\Ga^+}\bigl\vert\ga^+\cap\La\bigr\vert^n \,d\mu^+(\ga^+)= \int_{\Ga^2} \bigl\vert\ga^+\cap\La\bigr\vert^n \bigl\vert\ga^-\cap\La\bigr\vert^0\, d\mu(\ga^+,\ga^-)<\infty,
\]
implies $\mu^+\in\Mfm$.

Then, \eqref{cfoverproj} yields that, for $\la_z$-a.a. $\eta^+\in\Ga^+_{\La}$ and for any $\La\in\Bb$,
\begin{equation}\label{corfunc+}
k^+_\mu(\eta^+)=\int_{\Ga^+_{\La}} \frac{d(\mu^+)^{\La}}{d\la_z^\La}
(\eta^+\cup\xi^+) d\la_z(\xi^+).
\end{equation}

Put, in \eqref{gencorfunc}, $\eta^-=\emptyset$; then, by \eqref{RNprj} and~\eqref{corfunc+}, one get
\begin{align}
k_\mu(\eta^+,\emptyset)=&\int_{\Ga^+_{\La}}\biggl(\,\int_{\Ga^-_{\La}}
\frac{d\mu^{\La}}{d\la^2_z} (\eta^+\cup\xi^+,\xi^-)
d\la_z(\xi^-)\biggr)d\la_z(\xi^+)\notag\\
=&\int_{\Ga^+_{\La}} \frac{d(\mu^{\La})^+}{d\la_z}
(\eta^+\cup\xi^+) d\la_z(\xi^+)=k^+_\mu(\eta^+).
\end{align}
The proof for $k^-_\mu$ is the same.
\end{proof}

\begin{remark}
 Proposition \ref{prop:intK2G} implies that if $k_\mu$ is the correlation functional of some measure $\mu\in\MFm$, then $k_\mu$ is positive definite in the sense of Lenard.
\end{remark}

The recall about the following convolution between measurable functions $G_1$ and~$G_2$ on $\Ga_0^2$ (for details see \cite{Fin2012a}).
\begin{equation}\label{star2}
 (G_1\Star G_2)(\eta^+,\eta^-):=
 \sum_{\substack{\xi^+_1\sqcup\xi^+_2\sqcup\xi^+_3=\eta^+\\
 \xi^-_1\sqcup\xi^-_2\sqcup\xi^-_3=\eta^-}}
 G_1(\xi_1^+\cup\xi_2^+,\xi_1^-\cup\xi_2^-)\,
 G_2(\xi_2^+\cup\xi_3^+,\xi_2^-\cup\xi_3^-).
\end{equation}

By Proposition \ref{Fourier} and \eqref{redefK2}, \eqref{star2}, we obtain that, for any $G_1, G_2\in L_\ls^0(\Ga_0^2)$,
 \begin{equation}\label{eq:Fourier2}
 \bigl(\KK(G_1\Star G_2)\bigr)(\ga^+,\ga^-)
 =(\KK G_1)(\ga^+,\ga^-)\cdot (\KK G_2)(\ga^+,\ga^-), \quad (\ga^+,\ga^-)\in\Ga^2.
 \end{equation}
\begin{remark}
Similarly to Remark \ref{rem:posdefstar}, one get the following. By \eqref{eq:Fourier2}, $\KK(G\Star G)=|\KK G|^2\geq0$; then the positive definiteness in the sense of Lenard implies the positive definiteness in the sense of the $\Star$-convolution (for the definition of the latter see \cite{Fin2012a}).
\end{remark}

\section{Convolution of measures}
\subsection{Main properties}
During this Section we will use the following notations.
Let $F:\Ga\to \R$ be a measurable function. Consider the measurable function $\widetilde{F}:\Ga^2\to \R$ given by
$\widetilde{F}(\ga^+,\ga^-)= F(\ga^+\cup\ga^-)$, $(\ga^+,\ga^-)\in\Ga^2$.
Let $\mu_i\in\Mfm$, $i=1,2$. Consider the measure $\widehat{\mu}$ on~$\bigl(\Ga^2,\B(\Ga^2)\bigr)$ given by $d\widehat{\mu}(\ga^+,\ga^-)=d\mu_1(\ga^+)\,d\mu_2(\ga^-)$. On other words, $\widehat{\mu}=\mu_1\otimes\mu_2$. Clearly,~$\widehat{\mu}\in\MFm$.

\begin{definition}
    Let $\mu_i\in\M^1(\Ga)$, $i=1,2$.
    A measure $\mu\in\M^1(\Ga)$ is said to be the convolution of the measures $\mu_1$ and~$\mu_2$ if, for any measurable function $F:\Ga\to \R$, such that $\widetilde{F}\in
    L^1(\Ga^2,d\widehat{\mu})$, the following equality holds true
\begin{equation}\label{convofmeasGa}
            \int_\Ga F(\ga)d\mu(\ga)
        =\int_{\Ga^2}\widetilde{F}(\ga^+,\ga^-)d\widehat{\mu}(\ga^+,\ga^-)
        =\int_{\Ga^+}\int_{\Ga^-}F(\ga^{+}\cup \ga^-) \,d\mu_1(\ga^+)\,d\mu_2(\ga^-).
\end{equation}
    We denote this by $\mu=\mu_1\ast\mu_2$.
\end{definition}

\begin{proposition}
Let $\mu_i\in\Mfm$, $i=1,2$ and~$\mu=\mu_1\ast\mu_2$. Then $\mu\in\Mfm$. If, additionally, $\mu_i\in\Mpi$, $i=1,2$, then ~$\mu\in\Mpi$.
\end{proposition}
\begin{proof}
For any $\La\in\Bb$, $n\in\N$, one has
\begin{align*}
\int_\Ga |\ga_\La|^n\,d\mu(\ga)&=
\int_{\Ga^+}\int_{\Ga^-}\bigl(|\ga^+_\La|+|\ga^-_\La|\bigr)^n
\,d\mu_1(\ga^+)\,d\mu_2(\ga^-)\\
&=\sum_{k=0}^n \binom{n}{k}\int_\Ga |\ga_\La|^k\,d\mu_1(\ga)
\int_\Ga |\ga_\La|^{n-k}\,d\mu_2(\ga)<\infty,
\end{align*}
which proves the first statement. Next, for any $\B_\La(\Ga)$-measurable function $F$, one has
\begin{align*}
  \int_{\Ga_\La} F(\ga)\,d\mu^\La(\ga)&=  \int_{\Ga} F(\ga)\,d\mu(\ga)
  =\int_{\Ga^+}\int_{\Ga^-}F(\ga^+\cup\ga^-)\,d\mu_1(\ga^+)\,d\mu_2(\ga^-)
  \\&=\int_{\Ga^+_\La}\int_{\Ga^-_\La}F(\ga^+\cup\ga^-)\,d\mu_1^\La(\ga^+)\,d\mu_2^\La(\ga^-)
    \\&=\int_{\Ga^+_\La}\int_{\Ga^-_\La}F(\ga^+\cup\ga^-)\frac{d\mu_1^\La}{d\la}(\ga^+)
    \frac{d\mu_2^\La}{d\la}(\ga^-)\,d\la(\ga^+)\,d\la(\ga^-)
    \\&=\int_{\Ga_\La}F(\ga)\biggl(\frac{d\mu_1^\La}{d\la}\ast
    \frac{d\mu_2^\La}{d\la}\biggr)(\ga)\,d\la(\ga),
\end{align*}
where we used \eqref{minlosid-ast}. This prove the second statement as well.
\end{proof}

The following proposition describes the connection between convolutions of measures on the spaces $\Ga$ and~$\Ga_0$.
\begin{proposition} \label{connection}
Let $\mu_i\in\Mpi$ and $\rho_i$ be the corresponding correlation measures, $i=1,2$.
Then $\rho:=\rho_1\ast\rho_2$ is the correlation measure for $\mu:=\mu_1\ast\mu_2$.
\end{proposition}
\begin{proof}
Let $G\in\Bbs$, then, obviously, $\widetilde{G}\in\BBs\subset L^1\bigl(\Ga_0^2,d\widehat{\rho}\bigr)$.
Let $F=KG$. Then, for any $(\ga^+,\ga^-)\in\widetilde{\Ga}^2$ (i.e., $\ga^+\cap\ga^-=\emptyset$), we obtain
\begin{align}
\widetilde{F}(\ga^+,\ga^-)&=F(\ga^+\cup\ga^-)=\sum_{\eta\Subset\ga^+\cup\ga^-}G(\eta)\notag\\
&=\sum_{\eta^+\Subset\ga^+}\sum_{\eta^-\Subset\ga^-}G(\eta^+\cup\eta^-)
=\sum_{\eta^+\Subset\ga^+}\sum_{\eta^-\Subset\ga^-}\widetilde{G}(\eta^+,\eta^-)=\bigl(\KK\widetilde{G}\bigr)(\ga^+,\ga^-).\label{ex::1}
\end{align}
We prove now that $\widehat{\rho}=\rho_1\otimes\rho_2\in\MLf$ is the correlation measure for $\widehat{\mu}=\mu_1\otimes\mu_2\in\MFm$. To do this, let us check \eqref{cormeas2} with $\mu=\widehat{\mu}$, $\rho_\mu=\widehat{\rho}$. Namely, for any $A=A^+\times A^-\in\BBg$ with $A^\pm\in\Bbg$, one has
\begin{align*}
  &\int_{\Ga^2}(\KK\1_A)(\ga^+,\ga^-)\,d\widehat{\mu}(\ga^+,\ga^-)=\int_{\Ga^2}\sum_{\eta^+\Subset\ga^+}\sum_{\eta^-\Subset\ga^-} \1_A (\eta^+,\eta^-)\,d\mu_1(\ga^+)\,d\mu_2(\ga^-)\\
  =&\int_{\Ga^+}\sum_{\eta^+\Subset\ga^+} \1_{A^+} (\eta^+)\,d\mu_1(\ga^+)
  \int_{\Ga^-}\sum_{\eta^-\Subset\ga^-} \1_{A^-} (\eta^-)\,d\mu_2(\ga^-)=\rho_1(A^+)\rho_2(A^-)=\widehat{\rho}(A).
\end{align*}
Therefore, by~\eqref{cormeas2}, the measure $\widehat{\rho}$ coincides with the correlation measure for $\widehat{\mu}$, at least on all sets of the form $A=A^+\times A^-\in\BBg$ with~$A^\pm\in\Bbg$.
Hence, these measures are coincided on the whole class of sets $\BBg$. Since $\mu_i\in\Mpi$, $i=1,2$, the measure $\widehat{\mu}$ is locallz absolutelz continuous w.r.t. $\pi_z\otimes\pi_z$, $z>0$. Then, by Proposition~\ref{goodset}, $\widehat{\mu}(\widetilde{\Ga}^2)=1$. As a result,
\begin{align*}
\int_{\Ga_0} G(\eta)d\rho(\eta)&=\int_{\Ga_0^+}
\int_{\Ga_0^-}\widetilde{G}(\eta^{+},\eta^-)
\,d\widehat{\rho}(\eta^+,\eta^-)=\int_{\Ga^+}\int_{\Ga^-}\bigl(\KK\widetilde{G}\bigr)(\ga^{+},\ga^-)
\,d\widehat{\mu}(\ga^+,\ga^-)\\&=\iint_{\widetilde{\Ga}^2}
\bigl(\KK\widetilde{G}\bigr)(\ga^{+},\ga^-)\,d\widehat{\mu}(\ga^+,\ga^-)=\iint_{\widetilde{\Ga}^2}\widetilde{F}(\ga^{+},\ga^-)
\,d\widehat{\mu}(\ga^+,\ga^-)
\\&=\int_{\Ga^+} \int_{\Ga^-}\widetilde{F}(\ga^{+},\ga^-)
\,d\mu_1(\ga^+)\,d\mu_2(\ga^-)=\int_\Ga F(\ga)\,d\mu(\ga),
\end{align*}
which proves the statement.
\end{proof}

\begin{theorem}\label{critposdef}
Let functions $k_i:\Ga_0\to \R$, $i=1,2$ be measurable. Then the function $k(\eta)=(k_1\ast k_2)(\eta)$ is positive definite in the sense of Lenard on~$\Ga_0$ if only the function $\widehat{k}(\eta^+,\eta^-):=k_1(\eta^+)k_2(\eta^-)$ is positive definite in the sense of Lenard on~$\Ga_0^2$.
\end{theorem}
\begin{proof}
The equality \eqref{ex::1} yields that, if $G\in\Bbs$ and $KG\geq0$, then $\KK\widetilde{G}\geq0$. By \eqref{minlosid-ast}, one has
\[
\int_{\Ga_0}
G(\eta)(k_1\ast k_2)(\eta)d\la(\eta)=
\int_{\Ga_0^2}\widetilde{G}(\eta^+,\eta^-)\widehat{k}(\eta^+,\eta^-)d\la(\eta^+)d\la(\eta^-).
\]
Therefore, the positive definiteness of $\widehat{k}$ in the sense of Lenard on~$\Ga_0^2$ implies the positive definiteness of $k$ in the sense of Lenard on~$\Ga_0$.
\end{proof}

\subsection{Convolution of Gibbs measures}
Let $L_+^0(\Ga\times X)$ denote the class of all measurable nonnegative functions
$f:\Ga\times X \rightarrow \R_+$. Let a function $r\in L_+^0(\Ga\times X)$ be arbitrary and fixed. A measure $\mu\in{\mathcal{M}}_{\mathrm{fm}}^1(\Gamma)$ is said to be the Gibbs measure corresponding to the relative energy density (a.k.a. Papangelou intensity) $r$ iff, for any $h\in L_+^0(\Ga\times X)$, the following Georgii--Nguyen--Zessin identity holds true (cf. \cite{NZ1979}):
\begin{equation}  \label{Campbell}
\int_\Gamma\sum_{x\in\gamma}h(\gamma,x)d\mu(\gamma)=\int_\Gamma\int_{X}
h(\gamma\cup x)r(\gamma,x)dxd\mu(\gamma).
\end{equation}
Let $\M_\mathrm{fm}^1(\Ga;r)$ denote the class of all such measures. For properties of these measures and references see e.g. \cite{FK2005}. It should be underline that with a necessity \eqref{Campbell} yields
\begin{equation}\label{CCI}
    r(\gamma\cup y,x)r(\gamma,y)= r(\gamma\cup x,y)r(\gamma,x),
\end{equation}
for $\mu\times dx\times dy$-a.a. $(\ga,x,y)\in\Ga\times{X}\times{X}$.

\begin{proposition}\label{PropConvGibbs}
Let $\{r_1, r_2, r\}\subset L_+^0(\Ga\times X)$. Consider measures $\mu_i\in\M^1_\mathrm{fm}(\Ga; r_i)$, $i=1,2$, $\mu\in \M^1_\mathrm{fm}(\Ga,r)$. Let $\mu=\mu_1\ast\mu_2$. Then, for $\mu _{1}\times \mu _{2}\times dx$-a.a. $\left( \gamma ^{+},\gamma
^{-},x\right)\in\Ga\times\Ga\times X $, the following identity holds
\begin{equation}\label{additive}
r\left( \gamma ^{+}\cup \gamma ^{-},x\right) =r_{1}\left( \gamma
^{+},x\right) +r_{2}\left( \gamma ^{-},x\right) .
\end{equation}
\end{proposition}
\begin{proof}
By \eqref{convofmeasGa}, for any $h\in L_+^0(\Ga\times X)$, one has
\begin{align*}
& \quad \int_{\Gamma}\sum_{x\in \gamma}h\left( \gamma ,x\right) d\mu
\left(
\gamma \right) \\
& =\int_{\Gamma ^{+}}\int_{\Gamma ^{-}}\sum_{x\in \gamma ^{+}\cup \gamma
^{-}}h\left( \gamma ^{+}\cup \gamma ^{-},x\right) d\mu _{1}\left( \gamma
^{+}\right) d\mu _{2}\left( \gamma ^{-}\right) \\
& =\int_{\Gamma ^{+}}\int_{\Gamma ^{-}}\int_{{X}}h\left( \gamma
^{+}\cup x\cup \gamma ^{-},x\right) r_{1}\left( \gamma ^{+},x\right)
dxd\mu _{1}\left(
\gamma ^{+}\right) d\mu _{2}\left( \gamma ^{-}\right) \\
& \quad +\int_{\Gamma ^{+}}\int_{\Gamma ^{-}}\int_{{X}}h\left( \gamma
^{+}\cup \gamma ^{-}\cup x,x\right) r_{2}\left( \gamma ^{-},x\right)
dxd\mu
_{1}\left( \gamma ^{+}\right) d\mu _{2}\left( \gamma ^{-}\right) \\
& =\int_{\Gamma ^{+}}\int_{\Gamma ^{-}}\int_{{X}}h\left( \gamma
^{+}\cup x\cup \gamma ^{-},x\right) \left( r_{1}\left( \gamma
^{+},x\right) +r_{2}\left( \gamma ^{-},x\right) \right) dxd\mu
_{1}\left( \gamma ^{+}\right) d\mu _{2}\left( \gamma ^{-}\right) .
\end{align*}%
On the other hand,
\begin{align*}
& \quad \int_{\Gamma}\sum_{x\in \gamma}h\left( \gamma ,x\right) d\mu
\left(
\gamma \right) =\int_{\Gamma}\int_{{X}}h\left( \gamma \cup x,x\right) r\left(
\gamma
,x\right) dxd\mu \left( \gamma \right) \\
& =\int_{\Gamma ^{+}}\int_{\Gamma ^{-}}\int_{{X}}h\left( \gamma
^{+}\cup x\cup \gamma ^{-},x\right) r\left( \gamma ^{+}\cup \gamma
^{-},x\right) dxd\mu _{1}\left( \gamma ^{+}\right) d\mu _{2}\left(
\gamma ^{-}\right),
\end{align*}%
where we used \eqref{convofmeasGa} again.
Comparison of the getting expressions shows that \eqref{additive} holds true.
\end{proof}

\begin{remark}
It is easily seen from \eqref{additive} that, if only $\mu_i$ are Gibbs measures constructed by potentials $\Phi_i:\Ga_0\to\R$, $i=1,2$, i.e., $r_i(\ga,x)=\exp\Bigl\{-\sum_{\eta\Subset\ga}\Phi_i(\eta\cup x)\Bigr\}$, $i=1,2,$, then $\mu=\mu_1\ast\mu_2$ cannot be defined in a such way.
\end{remark}

\begin{corollary}
Let the condition~\ref{PropConvGibbs} holds. Let $\mu=\mu_1\ast\mu_2$. Then, for $\mu _{1}\times \mu _{2}\times dx\times dy$-a.a. $\left( \gamma ^{+}, \gamma
^{-}, x, y\right)\in\Ga\times\Ga\times X \times X$,
\begin{multline}
r_{1}\left( \gamma ^{+},x\right)r_{2}\left( \gamma ^{-},x\right)\left[ r_{1}\left( \gamma ^{+},x\right) r_{2}\left( \gamma
^{-},y\right) -r_{1}\left( \gamma ^{+},y\right) r_{2}\left( \gamma
^{-},x\right) \right] \\ \times \left[ r_{1}\left( \gamma ^{+}\cup
x,y\right) r_{2}\left( \gamma ^{-},y\right) -r_{2}\left( \gamma
^{-}\cup x,y\right) r_{1}\left( \gamma ^{+},y\right) \right] =0.\label{ness}
\end{multline}
\end{corollary}
\begin{proof}
By \eqref{CCI} (see also \cite{FK2005}), the expression
\begin{align*}
 r\left( \gamma ^{+}\cup \gamma ^{-}\cup y,x\right) r\left( \gamma ^{+}\cup
\gamma ^{-},y\right)
& =\left( r_{1}\left( \gamma ^{+},x\right) +r_{2}\left( \gamma ^{-}\cup
y,x\right) \right) \left( r_{1}\left( \gamma ^{+},y\right) +r_{2}\left(
\gamma ^{-},y\right) \right) \\
& =r_{1}\left( \gamma ^{+},x\right) r_{1}\left( \gamma ^{+},y\right)
+r_{1}\left( \gamma ^{+},x\right) r_{2}\left( \gamma ^{-},y\right) \\
& \quad +r_{2}\left( \gamma ^{-}\cup y,x\right) r_{1}\left( \gamma
^{+},y\right) +r_{2}\left( \gamma ^{-}\cup y,x\right) r_{2}\left( \gamma
^{-},y\right)
\end{align*}%
is a symmetric function of the variables $x$ and $y$, for $\mu _{1}\times \mu _{2}$-a.a. $\left(\gamma ^{+},\gamma ^{-}\right) $ and a.a. $x,y$.
But the expression $r_{2}\left( \gamma ^{-}\cup y,x\right) r_{2}\left( \gamma^{-},y\right)$ is also symmetric, therefore,
\begin{multline*}
r_{1}\left( \gamma ^{+},x\right) r_{2}\left( \gamma ^{-},y\right)
+r_{2}\left( \gamma ^{-}\cup y,x\right) r_{1}\left( \gamma
^{+},y\right) \\=r_{1}\left( \gamma ^{+},y\right) r_{2}\left( \gamma
^{-},x\right) +r_{2}\left( \gamma ^{-}\cup x,y\right) r_{1}\left(
\gamma ^{+},x\right) ,
\end{multline*}%
which yields
\begin{multline}
r_{1}\left( \gamma ^{+},x\right) r_{2}\left( \gamma ^{-},y\right)
-r_{1}\left( \gamma ^{+},y\right) r_{2}\left( \gamma ^{-},x\right)
\\=r_{2}\left( \gamma ^{-}\cup x,y\right) r_{1}\left( \gamma
^{+},x\right) -r_{2}\left( \gamma ^{-}\cup y,x\right) r_{1}\left(
\gamma ^{+},y\right) .\label{spec1}
\end{multline}%
On the other hand,
\begin{align*}
 r\left( \gamma ^{+}\cup \gamma ^{-}\cup y,x\right) r\left( \gamma ^{+}\cup
\gamma ^{-},y\right)
& =\left( r_{1}\left( \gamma ^{+}\cup y,x\right) +r_{2}\left( \gamma
^{-},x\right) \right) \left( r_{1}\left( \gamma ^{+},y\right) +r_{2}\left(
\gamma ^{-},y\right) \right) \\
& =r_{1}\left( \gamma ^{+}\cup y,x\right) r_{1}\left( \gamma ^{+},y\right)
+r_{1}\left( \gamma ^{+}\cup y,x\right) r_{2}\left( \gamma ^{-},y\right) \\
& \quad +r_{2}\left( \gamma ^{-},x\right) r_{1}\left( \gamma ^{+},y\right)
+r_{2}\left( \gamma ^{-},x\right) r_{2}\left( \gamma ^{-},y\right) .
\end{align*}%
Then, in the same way as above, one get
\begin{multline*}
r_{2}\left( \gamma ^{-},x\right) r_{1}\left( \gamma ^{+},y\right)
+r_{1}\left( \gamma ^{+}\cup y,x\right) r_{2}\left( \gamma ^{-},y\right) \\
=r_{2}\left( \gamma ^{-},y\right) r_{1}\left( \gamma ^{+},x\right)
+r_{1}\left( \gamma ^{+}\cup x,y\right) r_{2}\left( \gamma ^{-},x\right).
\end{multline*}%
Therefore,
\begin{multline}
r_{1}\left( \gamma ^{+},x\right) r_{2}\left( \gamma ^{-},y\right)
-r_{1}\left( \gamma ^{+},y\right) r_{2}\left( \gamma ^{-},x\right)
\\=r_{1}\left( \gamma ^{+}\cup y,x\right) r_{2}\left( \gamma
^{-},y\right) -r_{1}\left( \gamma ^{+}\cup x,y\right) r_{2}\left(
\gamma ^{-},x\right) .\label{spec2}
\end{multline}
Comparison of the right hand sides of \eqref{spec1} and \eqref{spec2} shows that
\begin{multline}
r_{1}\left( \gamma ^{+}\cup y,x\right) r_{2}\left( \gamma ^{-},y\right)
-r_{1}\left( \gamma ^{+}\cup x,y\right) r_{2}\left( \gamma ^{-},x\right)
\\
=r_{2}\left( \gamma ^{-}\cup x,y\right) r_{1}\left( \gamma ^{+},x\right)
-r_{2}\left( \gamma ^{-}\cup y,x\right) r_{1}\left( \gamma ^{+},y\right) .\label{spec3}
\end{multline}%
Let us multiply now the both sides of the equality \eqref{spec3} on $r_{1}\left( \gamma ^{+},y\right)
r_{1}\left( \gamma ^{+},x\right) r_{2}\left( \gamma ^{-},x\right)
r_{2}\left( \gamma ^{-},y\right) $. We obtain
\begin{align*}
& r_{1}\left( \gamma ^{+},y\right) r_{1}\left( \gamma ^{+},x\right)
r_{2}\left( \gamma ^{-},x\right) r_{2}\left( \gamma ^{-},y\right)
r_{1}\left( \gamma ^{+}\cup y,x\right) r_{2}\left( \gamma ^{-},y\right)  \\
& \quad -r_{1}\left( \gamma ^{+},y\right) r_{1}\left( \gamma ^{+},x\right)
r_{2}\left( \gamma ^{-},x\right) r_{2}\left( \gamma ^{-},y\right)
r_{1}\left( \gamma ^{+}\cup x,y\right) r_{2}\left( \gamma ^{-},x\right)  \\
& =r_{1}\left( \gamma ^{+},y\right) r_{1}\left( \gamma ^{+},x\right)
r_{2}\left( \gamma ^{-},x\right) r_{2}\left( \gamma ^{-},y\right)
r_{2}\left( \gamma ^{-}\cup x,y\right) r_{1}\left( \gamma ^{+},x\right)  \\
& \quad -r_{1}\left( \gamma ^{+},y\right) r_{1}\left( \gamma ^{+},x\right)
r_{2}\left( \gamma ^{-},x\right) r_{2}\left( \gamma ^{-},y\right)
r_{2}\left( \gamma ^{-}\cup y,x\right) r_{1}\left( \gamma ^{+},y\right) .
\end{align*}%
Therefore,
\begin{align*}
& r_{1}\left( \gamma ^{+}\cup x,y\right) r_{1}\left( \gamma ^{+},x\right)
r_{2}\left( \gamma ^{-},y\right) r_{2}\left( \gamma ^{-},x\right)  \left[ r_{1}\left( \gamma ^{+},x\right) r_{2}\left( \gamma
^{-},y\right) -r_{1}\left( \gamma ^{+},y\right) r_{2}\left( \gamma
^{-},x\right) \right]  \\
& =r_{2}\left( \gamma ^{-}\cup x,y\right) r_{2}\left( \gamma ^{-},x\right)
r_{1}\left( \gamma ^{+},y\right) r_{1}\left( \gamma ^{+},x\right)   \left[ r_{2}\left( \gamma ^{-},y\right) r_{1}\left( \gamma
^{+},x\right) -r_{2}\left( \gamma ^{-},x\right) r_{1}\left( \gamma
^{+},y\right) \right],
\end{align*}
which yields \eqref{ness}.
\end{proof}

The equality \eqref{ness} gives a `hint' about a proper class of relative energy densities $r_i(\ga,x)$ which satisfy \eqref{additive}. One may, for example, consider densities such that $r_i(\ga,x)=r_i(\ga)$, for $\mu_i\times dx$-a.a. $(\ga,x)\in\Ga\times{X}$. Then the expression in the first brackets in \eqref{ness} will be equal to $0$. Another useful variant is $r_i(\ga\cup y,x)=r_i(\ga,x)$, for $\mu_i\times dx\times dy$-a.a. $(\ga,x,y)\in\Ga\times{X}\times{X}$. Then the expression in the second brackets in \eqref{ness} will be equal to $0$.

\begin{example}
The so-called mixed Poisson measure corresponds to the both cases above. Namely, let
$p:(0;+\infty)\to(0;+\infty)$ with $\int_0^{\infty}p(z)dz=1$ and $p$ be continuous on $(0;+\infty)$. Consider the measure $\nu\in\M^1(\Ga)$ given by $\nu(A)=\int_0^\infty \pi_z(A) p(z)dz$, $A\in\B(\Ga)$. This is a mixture of Poisson measures with different intensities. By \eqref{MeckeID}, one has
\begin{align*}
\int_\Gamma\sum_{x\in\gamma}h(\gamma,x)d\nu(\gamma)&=
\int_0^\infty\int_\Gamma\sum_{x\in\gamma}h(\gamma,x)d\pi_z(\gamma)p(z)dz\\
&=\int_0^\infty z \int_\Gamma\int_{X} h(\gamma\cup x,x)dm(x) d\pi_z(\gamma) p(z)dz\\
&=\int_\Gamma\int_{X} h(\gamma\cup x,x)q(\ga,x)dm(x)  d\nu(\gamma).
\end{align*}
Here, for a.a. $z\in(0;+\infty)$, we consider $q(\ga,x)=z$, for $\pi_z$-a.a. $\ga\in\Ga$ and for all $x\in{X}$. More precisely, as it was noted before in Remark~\ref{remPi}, $\pi_{z_1}\bot\pi_{z_2}$ for $z_1\neq z_2$. On the other hand, by e.g. \cite{FU1998},
\begin{equation}\label{sf}
    \lim_{\La\uparrow{X}}\frac{|\ga\cap\La|}{m(\La)}=z \quad \text{for $\pi_z$-a.a. $\ga\in\Ga$, $z>0$.}
\end{equation}
Therefore, if $A_z$ is the set of configurations which satisfy \eqref{sf}, then $\pi_{z_1}(A_{z_2})=\delta_{z_1,z_2}$ (the Kronecker symbol). As a result, $\nu(A)=1$, where $A=\bigcup_{z>0}A_z$. Hence,
\begin{equation}\label{ada}
    q(\ga,x)=\lim_{\La\uparrow{X}}\frac{|\ga\cap\La|}{m(\La)} \quad \text{for $\nu$-a.a. $\ga\in\Ga$ and for all $x\in{X}$.}
\end{equation}
Stress that $q$ does not depend on $p$, i.e., the function $q$ does not define uniquely the measure $\nu$. Clearly, $q(\ga,x)=q(\ga\cup\eta,x)$, for all $\eta\in\Ga_0$, $\ga\cap\eta=\emptyset$. Similarly, $q(\ga_1\cup\ga_2,x)=q(\ga_1,x)+q(\ga_2,x)$, if only $\ga_1\cap\ga_2=\emptyset$. Therefore, if $\mu$, $\mu_1$, $\mu_2$ are mixed Poisson measures given by functions $p,p_1,p_2$, correspondingly, then the equality \eqref{additive} will holds true since $r_1=r_2=r=q$. To prove that a convolution of mixed Poisson measures is also a mixed Poisson measure we recall that (see e.g. \cite{AKR1998a}) any Poisson measure is uniquely defined by its values on sets $C(\La,n)=\bigl\{\ga\in\Ga\bigm||\ga\cap\La|=n\bigr\}$, $\La\in\Bb$, $n\in\N_0$ and, moreover,
\[
\pi_z(C(\La,n))=\frac{(zm(\La))^n}{n!}e^{-zm(\La)}.
\]
Then, by the definition of the convolution of measures, for $\mu=\mu_1\ast\mu_2$, one has
\begin{align*}
&\quad\int_{\Gamma }\1_{C\left(
\Lambda ,n\right) }\left( \gamma \right) d\mu \left( \gamma \right)
=\int_{\Gamma }\int_{\Gamma }\1_{C\left( \Lambda ,n\right) }\left( \gamma
_{1}\cup \gamma _{2}\right) d\mu _{1}\left( \gamma _{1}\right) d\mu
_{2}\left( \gamma _{2}\right)  \\
&=\int_{\Gamma }\int_{\Gamma }\1_
{\left\vert \gamma _{1}\cap \Lambda
\right\vert +\left\vert \gamma _{2}\cap \Lambda \right\vert =n}d\mu
_{1}\left( \gamma _{1}\right) d\mu _{2}\left( \gamma _{2}\right)  =\int_{\Gamma }\int_{\Gamma }\sum_{k=0}^{n}\1_{\left\vert \gamma _{1}\cap
\Lambda \right\vert =k}\1_{\left\vert \gamma _{2}\cap \Lambda \right\vert
=n-k}d\mu _{1}\left( \gamma _{1}\right) d\mu _{2}\left( \gamma _{2}\right)
\\
&=\sum_{k=0}^{n}\int_{\Gamma }\1_{\left\vert \gamma _{1}\cap \Lambda
\right\vert =k}d\mu _{1}\left( \gamma _{1}\right) \int_{\Gamma
}\1_{\left\vert \gamma _{2}\cap \Lambda \right\vert =n-k}d\mu _{2}\left(
\gamma _{2}\right)  \\
&=\sum_{k=0}^{n}\frac{1}{k!}\frac{1}{\left( n-k\right) !}\int_{0}^{\infty
}\left( zm\left( \Lambda \right) \right) ^{k}e^{-zm\left( \Lambda \right)
}p_{1}\left( z\right) dz\int_{0}^{\infty }\left( zm\left( \Lambda \right)
\right) ^{n-k}e^{-zm\left( \Lambda \right) }p_{2}\left( z\right) dz \\
&=\frac{\left( m\left( \Lambda \right) \right) ^{n}}{n!}\sum_{k=0}^{n}\frac{%
n!}{k!\left( n-k\right) !}\int_{0}^{\infty }\int_{0}^{\infty
}z_{1}^{k}z_{2}^{n-k}e^{-\left( z_{1}+z_{2}\right) m\left( \Lambda \right)
}p_{1}\left( z_{1}\right) dz_{1}p_{2}\left( z_{2}\right) dz_{2} \\
&=\frac{\left( m\left( \Lambda \right) \right) ^{n}}{n!}\int_{0}^{\infty
}\int_{0}^{\infty }\left( z_{1}+z_{2}\right) ^{n}e^{-\left(
z_{1}+z_{2}\right) m\left( \Lambda \right) }p_{1}\left( z_{1}\right)
p_{2}\left( z_{2}\right) dz_{1}dz_{2} \\
&=\int_{0}^{\infty }e^{-zm\left( \Lambda \right) }\frac{(zm(\La))^{n}}{n!}%
\int_{0}^{\infty }p_{1}\left( z_{1}\right) p_{2}\left( z-z_{1}\right) dz_{1}dz.
\end{align*}%
Hence, $\mu$ is the Poisson measure given by the function $p(z)=\int_{0}^{\infty }p_{1}\left( z_{1}\right) p_{2}\left( z-z_{1}\right) dz_{1}$, i.e., $p=p_1*p_2$ in the sense of the usual convolution on the real line.

To summarize, a convolution of two mixed Poisson measures is a mixed Poisson measure, these measure are Gibssian in the sense of \eqref{Campbell}, and their their relative energies densities are defined by \eqref{ada}.
\end{example}

\section{Invariant measures and examples of derivative operators w.r.t. the $*$-convolutions of functions on $\Ga_0$}

Recall that a measure $\mu\in\Mfm$ is said to be invariant for an operator $L$ which is defined on a class of functions on $\Gamma$ if for any such a $F$ the following equality holds true
\begin{equation}\label{inv}
\int_\Gamma (LF)(\gamma)d\mu(\gamma)=0.
\end{equation}
 Suppose that, for any $F\in\Fcyl$, $|LF(\eta)|<\infty$, at least for all $\eta\in\Ga_0$. Then, for any $G\in\Bbs$, we get $F=KG\in\Fcyl$ and the expression $K^{-1}LF$, given by analogy with \eqref{k-1trans}, is point-wise defined (for details see Section~3). As a result, one can consider the operator
\[
\hat{L}G:=K^{-1}LKG, \quad G\in\Bbs.
\]

Suppose that the integral in the left hand side of \eqref{inv} is finite for any $F\in\Fcyl$ and, additionally, that the correlation functional $k_\mu$ of the measure $\mu$ exists. Then, by \eqref{GintEqKGint-cf}, one has
\[
\langle\!\langle \hat{L}G, k_\mu\rangle\!\rangle=\int_{\Ga_0} (\hat{L}G)(\eta)k_\mu(\eta)\,d\la(\eta)=0,
\]
for all $G\in\Bbs$ (see also \cite[Section 4]{Fin2012a}). As a result, one can consider the equation for the dual operator
\begin{equation}  \label{invcorfunc}
\widehat{L}^\ast k=0,
\end{equation}
which might be considered either in a weak sense or e.g. in the space $\K_{C,\delta}$ (for details see also \cite{FKK2011a}). On the other hand, the equation \eqref{inv} does not define a measure uniquely. Indeed, let $\mu_i$, $i=1,2$ be invariant measures w.r.t. an operator $L$ (in particular, $\mu_1=\mu_2$), let $k_{1,2}$ be the corresponding correlation functionals.
By Proposition~\ref{connection} and \cite[Proposition 5.3]{Fin2012a}, the function
$k=k_1\ast k_2$ is the correlation functional of
$\mu:=\mu_1\ast\mu_2$. Suppose now that the operator
$\widehat{L}^\ast$ is a derivative operator w.r.t. the convolution \eqref{ast-q} (see \cite[Seubsection~5.3]{Fin2012a}). Then
\begin{equation*}
\widehat{L}^\ast k=(\widehat{L}^\ast k_1)\ast
k_2+k_1\ast(\widehat{L}^\ast k_2) =0,
\end{equation*}
therefore, $\mu=\mu_1\ast\mu_2$ is also an invariant measure for the operator $L$. In particular, $\mu_1^{*n}$ will be an invariant measure for the operator $L$, for all $n\in\N$.

\begin{example}
Consider the operator
\begin{equation}\label{CM-gen}
    (L_{\mathrm{CM}}F)(\ga)=\sum_{x\in\ga}\bigl[F(\ga\setminus x)-F(\ga)\bigl]+\sum_{y\in\ga}\int_{\R^d}a(x-y)\bigl[F(\ga\cup x)-F(\ga)\bigr]dx,
\end{equation}
where $0\leq a\in L^1(\R^d)$, $a(-x)=a(x)$, $x\in\R^d$. This is the generator of the so-called contact model introduced in \cite{KS2006}. In \cite{KKP2008}, it was shown that there exists a family of invariant measures $\mu_{\mathrm{inv}}$ which are parameterized by their first correlation functions $k_{\mathrm{inv}}^{(1)}=c$, $c>0$. Therefore, for each $c_{i}>0$, $i=1,2$, there exist two invariant measures $\mu_{1,2}$ for the operator \eqref{CM-gen}. Consider the measure $\mu=\mu_1\ast\mu_2$. By previous considerations, its the first correlation function is equal to $c_1+c_2$ and, moreover, this measure is invariant for the operator ${L}_{\mathrm{CM}}$ provided that the operator $\widehat{L}_{\mathrm{CM}}^*$ is a derivative operator w.r.t. the convolution \eqref{ast-q}. Below we prove the latter fact as a particular case of much more general situation.
\end{example}

Consider the following two operators
\begin{align*}
\left( L_{-}F\right) \left( \gamma \right)  &=\sum_{x\in \gamma
}\int_{\Gamma _{0}}d\left( x,\omega \right) \left[ F\left( \gamma \setminus
x\cup \omega \right) -F\left( \gamma \right) \right] d\lambda \left( \omega
\right) , \\
\left( L_{+}F\right) \left( \gamma \right)  &=\sum_{x\in \gamma
}\int_{\Gamma _{0}}b\left( x,\omega \right) \left[ F\left( \gamma \cup
\omega \right) -F\left( \gamma \right) \right] d\lambda \left( \omega
\right)
\end{align*}%
(here and subsequently, we write just $x$ instead of $\{x\}$).
Let functions $b$ and $d$ be measurable and nonnegative on ${X}\times\Ga_0$ and, additionally, suppose that $\int_{\Ga_\La}\bigl(b(x,\omega)+d(x,\omega)\bigr)d\la(\omega)<\infty$. Then $|LF(\eta)|<\infty$, $\eta\in\Ga_0$.
Let us denote
\begin{align*}
\bigl( L_{-}^{\left( n\right) }F\bigr) \left( \gamma \right)  &=\sum_{x\in \gamma }\int_{X^{n}}d\left( x,\left\{ y_{1},\ldots
,y_{n}\right\} \right) \left[ F\left( \gamma \setminus x\cup \left\{
y_{1},\ldots ,y_{n}\right\} \right) -F\left( \gamma \right) \right]
dy_{1}\ldots dy_{n}, \\
\bigl( L_{+}^{\left( n\right) }F\bigr) \left( \gamma \right)  &=\sum_{x\in \gamma }\int_{X^{n}}d\left( x,\left\{ y_{1},\ldots
,y_{n}\right\} \right) \left[ F\left( \gamma \cup \left\{ y_{1},\ldots
,y_{n}\right\} \right) -F\left( \gamma \right) \right] dy_{1}\ldots dy_{n},
\end{align*}%
i.e., $L_\pm=\sum_{n=0}^\infty \frac{1}{n!} L_\pm^{(n)}$.
For example,
\begin{align*}
\bigl( L_{-}^{\left( 0\right) }F\bigr) \left( \gamma \right)  &=\sum_{x\in
\gamma }d\left( x,\emptyset \right) \left[ F\left( \gamma \setminus x\right)
-F\left( \gamma \right) \right] , \\
\bigl( L_{-}^{\left( 1\right) }F\bigr) \left( \gamma \right)  &=\sum_{x\in
\gamma }\int_{X}d\left( x,y\right) \left[ F\left( \gamma \setminus x\cup
y\right) -F\left( \gamma \right) \right] dy, \\
\bigl( L_{+}^{\left( 0\right) }F\bigr) \left( \gamma \right)  &=0, \\
\bigl( L_{+}^{\left( 1\right) }F\bigr) \left( \gamma \right)  &=\sum_{x\in
\gamma }\int_{X}d\left( x,y\right) \left[ F\left( \gamma \cup y\right)
-F\left( \gamma \right) \right] dy.
\end{align*}%
In particular, if $d(x,\emptyset)\equiv=1$, $d(x,y)=a(x-y)$, then $L_\mathrm{CM}= L_{-}^{(0)}+L_{+}^{(1)}$.

We proceed to the calculation of $\hat{L}_\pm$. One has
\begin{align*}
\left( KG\right) \left( \gamma \setminus x\cup \omega \right) -\left(
KG\right) \left( \gamma \right)
&=\sum_{\eta \Subset \gamma \setminus x}\sum_{\emptyset \neq \zeta \subset
\omega }G\left( \eta \cup \zeta \right) -\sum_{\eta \Subset \gamma \setminus
x}G\left( \eta \cup x\right)  \\
&=K\biggl[ \,\sum_{\emptyset \neq \zeta \subset \omega }G\left( \cdot \cup
\zeta \right) -G\left( \cdot \cup x\right) \biggr] \left( \gamma \setminus
x\right) ,
\end{align*}%
then
\begin{align*}
&\quad(\hat{L}_{-}G)(\eta)=\left( K^{-1}\left( L_{-}KG\right) \right) \left( \eta \right)
\\&=\sum_{x\in \eta }\int_{\Gamma _{0}}d\left( x,\omega \right) \biggl[
\,\sum_{\emptyset \neq \zeta \subset \omega }G\left( \eta \setminus x\cup
\zeta \right) -G\left( \eta \right) \biggr] d\lambda \left( \omega \right)
\\
&=-\sum_{x\in \eta }d\left( x\right) G\left( \eta \right) -\sum_{x\in \eta
}d\left( x\right) G\left( \eta \setminus x\right) +\sum_{x\in \eta
}\int_{\Gamma _{0}}d\left( x,\omega \right) \sum_{\zeta \subset \omega
}G\left( \eta \setminus x\cup \zeta \right) d\lambda \left( \omega \right)
\\
&=-D\left( \eta \right) G\left( \eta \right) -\sum_{x\in \eta }d\left(
x\right) G\left( \eta \setminus x\right) +\sum_{x\in \eta }\int_{\Gamma
_{0}}\left( \int_{\Gamma _{0}}d\left( x,\omega \cup \xi \right) d\lambda
\left( \omega \right) \right) G\left( \eta \setminus x\cup \zeta \right)
d\lambda \left( \xi \right)  \\
&=-D\left( \eta \right) G\left( \eta \right) -\sum_{x\in \eta }d\left(
x\right) G\left( \eta \setminus x\right) +\sum_{x\in \eta }\int_{\Gamma
_{0}}d_{1}\left( x,\xi \right) G\left( \eta \setminus x\cup \zeta \right)
d\lambda \left( \xi \right) ,
\end{align*}%
where we denoted
\[
    d\left( x\right)  =\int_{\Gamma _{0}}d\left( x,\omega \right) d\lambda
    \left( \omega \right) , \quad D\left( \eta \right) =\sum_{x\in \eta }d\left(
    x\right) , \quad
    d_{1}\left( x,\xi \right)  =\int_{\Gamma _{0}}d\left( x,\omega \cup \xi
    \right) d\lambda \left( \omega \right) .
\]

Similarly, using the equality
\[
\left( KG\right) \left( \gamma \cup \omega \right) -\left( KG\right) \left(
\gamma \right) =\sum_{\eta \Subset \gamma }\sum_{\emptyset \neq \zeta
\subset \omega }G\left( \eta \cup \zeta \right) =K\biggl[ \,\sum_{\emptyset
\neq \zeta \subset \omega }G\left( \cdot \cup \zeta \right) \biggr] \left(
\gamma \right) ,
\]%
we derive
\begin{align*}
&\quad(\hat{L}_{+}G)(\eta)=\left( K^{-1}\left( L_{+}KG\right) \right) \left( \eta \right)\\
&=\sum_{x\in \eta }\int_{\Gamma _{0}}b\left( x,\omega \right)
\sum_{\emptyset \neq \zeta \subset \omega }G\left( \eta \setminus x\cup
\zeta \right) d\lambda \left( \omega \right) +\sum_{x\in \eta }\int_{\Gamma
_{0}}b\left( x,\omega \right) \sum_{\emptyset \neq \zeta \subset \omega
}G\left( \eta \cup \zeta \right) d\lambda \left( \omega \right)  \\
&=\sum_{x\in \eta }\int_{\Gamma _{0}}b\left( x_{1},\zeta \right) G\left(
\eta \setminus x\cup \zeta \right) d\lambda \left( \zeta \right) +\sum_{x\in
\eta }\int_{\Gamma _{0}}b_{1}\left( x,\zeta \right) G\left( \eta \cup \zeta
\right) d\lambda \left( \zeta \right)  \\
&\quad -\sum_{x\in \eta }G\left( \eta \setminus x\right) b\left( x\right)
-B\left( \eta \right) G\left( \eta \right)
\end{align*}%
where
\[
b\left( x\right)  =\int_{\Gamma _{0}}b\left( x,\omega \right) d\lambda
\left( \omega \right), \quad B\left( \eta \right) =\sum_{x\in \eta }b\left(
x\right) , \quad b_{1}\left( x,\xi \right)  =\int_{\Gamma _{0}}b\left( x,\omega \cup \xi
\right) d\lambda \left( \omega \right) .
\]

It is easily seen that the both expressions, which we obtained, are satisfied the equality
\begin{equation}  \label{dual-sum}
(\hat{L}_\pm G)(\eta\cup\xi) =\bigl((\hat{L}_\pm G)(\cdot\cup\xi)\bigr)(\eta)+\bigl((\hat{L}_\pm G)(\cdot\cup\eta)\bigr)(\xi).
\end{equation}
Therefore, by \cite[Proposition 5.8]{Fin2012a}, the corresponding operators $\hat{L}_\pm ^*$ will be derivative operators w.r.t. the convolution \eqref{ast-q}.

\def\cprime{$'$}

\end{document}